\documentclass[11pt,a4paper,reqno]{amsart}
\usepackage[english]{babel}
\usepackage[T1]{fontenc}
\usepackage{verbatim}
\usepackage{palatino}
\usepackage{amsmath}
\usepackage{mathabx}
\usepackage{amssymb}
\usepackage{amsthm}
\usepackage{amsfonts}
\usepackage{graphicx}
\usepackage{esint}
\usepackage{color}
\usepackage{mathtools}

\DeclarePairedDelimiter\floor{\lfloor}{\rfloor}
\usepackage[colorlinks = true, citecolor = black]{hyperref}
\title[Planar incidences and geometric inequalities in $\He$]{Planar incidences and geometric inequalities in the Heisenberg group}
\author{Katrin F\"assler, Tuomas Orponen, and Andrea Pinamonti}
\address{Department of Mathematics and Statistics \\ University of Jyv\"askyl\"a, P.O. Box. 35 (MaD), FI-40014 University of Jyv\"askyl\"a \\ Finland}
\address{Department of Mathematics and Statistics\\ University of Helsinki,
P.O. Box 68 (Pietari Kalmin katu 5)\\
FI-00014 University of Helsinki\\
Finland}
\address{Department of Mathematics\\
University of Trento,
Via Sommarive 14\\
I-38123 Povo\\
Italy}
\email{katrin.s.fassler@jyu.fi}
\email{tuomas.orponen@helsinki.fi}
\email{andrea.pinamonti@unitn.it}
\date{\today}
\thanks{K.F. is supported by the Academy of Finland via the project \emph{Singular integrals, harmonic functions, and boundary regularity in Heisenberg groups}, grant Nos. 321696, 328846. T.O. is supported by the Academy of Finland via the project \emph{Quantitative rectifiability in Euclidean and non-Euclidean spaces}, grant Nos. 309365, 314172.}

\newcommand{\R}{\mathbb{R}}
\newcommand{\W}{\mathbb{W}}
\newcommand{\He}{\mathbb{H}}
\newcommand{\N}{\mathbb{N}}

\newcommand{\Z}{\mathbb{Z}}
\newcommand{\tn}{\mathbb{P}}

\newcommand{\calH}{\mathcal{H}}

\newcommand{\spa}{\operatorname{span}}

\newcommand{\card}{\operatorname{card}}

\def\Barint_#1{\mathchoice
          {\mathop{\vrule width 6pt height 3 pt depth -2.5pt
                  \kern -8pt \intop}\nolimits_{#1}}%
          {\mathop{\vrule width 5pt height 3 pt depth -2.6pt
                  \kern -6pt \intop}\nolimits_{#1}}%
          {\mathop{\vrule width 5pt height 3 pt depth -2.6pt
                  \kern -6pt \intop}\nolimits_{#1}}%
          {\mathop{\vrule width 5pt height 3 pt depth -2.6pt
                  \kern -6pt \intop}\nolimits_{#1}}}

\numberwithin{equation}{section}

\theoremstyle{plain}
\newtheorem{thm}[equation]{Theorem}

\newtheorem{lemma}[equation]{Lemma}

\newtheorem{ex}[equation]{Example}
\newtheorem{cor}[equation]{Corollary}
\newtheorem{proposition}[equation]{Proposition}
\newtheorem{question}{Question}

\theoremstyle{definition}

\newtheorem{definition}[equation]{Definition}

\theoremstyle{remark}
\newtheorem{remark}[equation]{Remark}

\newcommand{\nref}[1]{(\hyperref[#1]{#1})}

\begin{document}

\begin{abstract} We prove that if $P,\mathcal{L}$ are finite sets of $\delta$-separated points and lines in $\R^{2}$, the number of $\delta$-incidences between $P$ and $\mathcal{L}$ is no larger than a constant times
\begin{displaymath} |P|^{2/3}|\mathcal{L}|^{2/3} \cdot \delta^{-1/3}. \end{displaymath}
We apply the bound to obtain the following variant of the \emph{Loomis-Whitney inequality} in the Heisenberg group:
\begin{displaymath} |K| \lesssim |\pi_{x}(K)|^{2/3} \cdot |\pi_{y}(K)|^{2/3}, \qquad K \subset \He. \end{displaymath}
Here $\pi_{x}$ and $\pi_{y}$ are the \emph{vertical projections} to the $xt$- and $yt$-planes, respectively, and $|\cdot|$ refers to natural Haar measure on either $\He$, or one of the planes. Finally, as a corollary of the Loomis-Whitney inequality, we deduce that
\begin{displaymath} \|f\|_{4/3} \lesssim \sqrt{\|Xf\| \|Yf\| }, \qquad f \in BV(\He), \end{displaymath}
where $X,Y$ are the standard horizontal vector fields in $\He$. This is a sharper version of the classical geometric Sobolev inequality $\|f\|_{4/3} \lesssim \|\nabla_{\He}f\|$ for $f \in BV(\He)$.
 \end{abstract}

\maketitle

\tableofcontents

\section{Introduction}

The \emph{Loomis-Whitney} inequality in $\R^{n}$ bounds the volume of a set $K \subset \R^{n}$ by the areas of its coordinate projections:
\begin{equation}\label{LWIneq} |K| \leq \prod_{j = 1}^{n} |\pi_{j}(K)|^{1/(n - 1)}, \end{equation}
where $\pi_{j}(x_{1},\ldots,x_{n}) = (x_{1},\ldots,x_{j - 1},x_{j + 1},\ldots,x_{n})$. Here $|A|$ refers to $d$-dimensional Lebesgue measure in $\R^{d}$ whenever $A \subset \R^{d}$. The same notation will also refer to cardinality, but the appropriate meaning should always be clear from the context. The inequality \eqref{LWIneq} is due to Loomis and Whitney \cite{MR0031538} from 1949. The starting point of the present paper was to find an analogue of \eqref{LWIneq} in Heisenberg groups. It turns out that already in the first group $\He = (\R^{3},\cdot)$, this leads to an interesting incidence geometric problem, which we are able to fully resolve. In higher groups, the question will remain open.

\subsection{A bound on $\delta$-incidences in the plane} Before setting up the Heisenberg notation, we discuss the more elementary incidence geometry problem -- in $\R^{2}$.
Let $P \subset Q_{0} := [-1,1]^{2}$ be a finite set, and let
$\mathcal{L}$ be a finite collection of lines $\ell_{(a,b)} := \{y
= ax + b\} \subset \R^{2}$, with $(a,b) \in Q_{0}$; we denote the
collection of all such lines by $\mathcal{Q}_{0}$. We fix a
"scale" $0 < \delta < 1$, and assume that all the points in $P$
and lines in $\mathcal{L}$ are $\delta$-separated. Two points $p,q
\in \R^{2}$ are called $\delta$-separated if $|p - q| \geq
\delta$. Two lines $\ell_{p},\ell_{q} \in \mathcal{Q}_{0}$ are
called $\delta$-separated if $p,q \in Q_{0}$ are
$\delta$-separated. We say that a point $p \in \R^{2}$ is
\emph{$\delta$-incident} to a line $\ell \subset \R^{2}$ if $p$
lies in the \emph{$\delta$-neighborhood} $\ell(\delta)$ of $\ell$.
We also write
\begin{displaymath} \mathcal{I}_{\delta}(P,\mathcal{L}) := \{(p,\ell) \in P \times \mathcal{L} : p \text{ is $\delta$-incident to } \ell\}. \end{displaymath}
Here is the first main result of the paper:
\begin{thm}\label{mainIntro} Let $P \subset Q_{0}$ and $\mathcal{L} \subset \mathcal{Q}_{0}$ be $\delta$-separated. Then,
\begin{displaymath} |\mathcal{I}_{\delta}(P,\mathcal{L})| \lesssim |P|^{2/3}|\mathcal{L}|^{2/3} \cdot \delta^{-1/3}. \end{displaymath}
Here $|\cdot|$ refers to cardinality on both sides of the inequality. The implicit constant is absolute.  \end{thm}
This estimate is a close relative of the Szemer\'edi-Trotter incidence bound \cite{MR729791} which, in our notation, says that $|\mathcal{I}_{0}(P,\mathcal{L})| \lesssim |P|^{2/3}|\mathcal{L}|^{2/3} + |P| + |\mathcal{L}|$ (without any hypotheses on the separation of $P$ or $\mathcal{L}$). The Szemer\'edi-Trotter bound for $|\mathcal{I}_{0}|$ is typically much better than the one in Theorem \ref{mainIntro} for $|\mathcal{I}_{\delta}|$, but this is to be expected. In fact, the bound in Theorem \ref{mainIntro} cannot be improved, unless one assumes stronger separation from either $P$ or $\mathcal{L}$. The simplest non-trivial sharpness example is perhaps given by letting $P$ be a $\delta$-packing in a tube of dimensions $\delta^{1/2} \times \delta$. All of the points in $P$ are $\delta$-incident to a $\delta$-separated family of lines of cardinality $\sim \delta^{-1/2}$, giving $|\mathcal{I}_{\delta}(P,\mathcal{L})| \sim \delta^{-1}$. This matches the upper bound in Theorem \ref{mainIntro}. More generally, sharpness examples are given by letting $P$ be a $\delta$-packing in a rectangle $R = [0,r] \times [0,s] \subset Q_{0}$, with $\delta \leq s \leq r \leq 1$, and letting $\mathcal{L} \subset \mathcal{Q}_{0}$ be a $\delta$-packing of lines meeting $R$.

We will infer Theorem \ref{mainIntro} from an estimate for the number of \emph{$k$-rich points} relative to an $\epsilon$-separated line family, with $\epsilon \geq \delta$, see Theorem \ref{mainRich}. To prove Theorem \ref{mainIntro}, only the case $\epsilon = \delta$ of Theorem \ref{mainRich} is needed; we decided to include the general case $\epsilon \geq \delta$ since the same question has been recently studied by Guth, Solomon, and Wang \cite{MR4034922}. We will comment on the differences between the results in Remark \ref{rem1}. Other related results in the plane are contained in \cite{2020arXiv200111304H,2020arXiv200102551L,MR4055989,2020arXiv200301636S}. In higher dimensions, the problem of bounding the number of $\delta$-incidences between points and lines is at the heart of Kakeya and restriction problems, see \cite{MR2275834,MR3454378,2019arXiv190910693G,MR3830894,2019arXiv190805589H,MR3881832,MR3868003} for a few recent papers.

\subsection{Loomis-Whitney and Gagliardo-Nirenberg-Sobolev inequalities in $\He$} We then move to the Heisenberg group, although we postpone most of the precise definitions to Section \ref{s:LW}. In brief, the first Heisenberg group $\He$ is $\R^{3}$ equipped with a non-commutative group law "$\cdot$" which makes it a nilpotent Lie group. The "vertical" planes in $\R^{3}$ containing the $t$-axis are subgroups of $\He$ -- known as the \emph{vertical subgroups}. To a vertical subgroup $\W \subset \He$, we associate the \emph{complementary horizontal subgroup} $\mathbb{L}$ which, as a subset of $\R^{3}$, is just the orthogonal complement of $\W$, a line in the $xy$-plane. For subsets of $\He \cong \R^{3}$, the notation $|\cdot|$ will refer to Lebesgue measure on $\R^{3}$, and for subsets of a vertical plane $\R^{2} \cong \W \subset \He$, the notation $|\cdot|$ will refer to Lebesgue measure in $\R^{2}$. All integrations on $\He$ or $\W$ will be performed with respect to these measures. Up to multiplicative constants, they could also be defined as the $4$-and $3$-dimensional Hausdorff measures, respectively, relative to a natural metric on $\He$. So, our measures coincide with canonical "intrinsic" objects in $\He$.

 Fixing a pair $(\W,\mathbb{L})$, as above, every point $p \in \He$ can be uniquely decomposed as $p = w \cdot v$, where $w \in \W$ and $v \in \mathbb{L}$. This operation gives rise to the \emph{vertical} and \emph{horizontal projections}
 \begin{displaymath} p \mapsto \pi_{\W}(p) := w \quad \text{and} \quad p \mapsto \pi_{\mathbb{L}}(p) := v. \end{displaymath}
 The vertical projections, in particular, play a significant role in the geometric measure theory of Heisenberg groups -- as do orthogonal projections in $\R^{n}$ -- so they have been actively investigated in recent years, see \cite{MR3047423,MR2955184,MR3992573,MR3495435,2018arXiv181112559H,2020arXiv200204789H}. The vertical projections are non-linear maps, but their \emph{fibres} $\pi_{\W}^{-1}\{w\}$ are nevertheless lines. In fact, the fibres of $\pi_{\W}$ are precisely the left translates of the line $\mathbb{L}$, that is, $\pi_{\W}^{-1}\{w\} = w \cdot \mathbb{L}$ for $w \in \W$.

With this introduction in mind, we are interested in proving a variant of the Loomis-Whitney inequality \eqref{LWIneq} for subsets of $\He$ in terms of the vertical projections $\pi_{\W}$. In $\R^{n}$, the inequality makes a reference to the $n$ coordinate projections. These are, now, best viewed as the projections whose fibres are translates of lines parallel to the coordinate axes. In $\He$, it seems natural to fix a basis for the $xy$-plane, say $e_{1} = (1,0,0)$ and $e_{2} = (0,1,0)$, and consider the two vertical projections $\pi_{1} := \pi_{\W_{1}}$ and $\pi_{2} := \pi_{\W_{2}}$ whose fibres are left translates of $\mathbb{L}_{1} := \spa(e_{1})$ and $\mathbb{L}_{2} := \spa(e_{2})$. The exact formulae are
\begin{displaymath} \pi_{1}(x,y,t) = (0,y,t + \tfrac{xy}{2}) \quad \text{and} \quad \pi_{2}(x,y,t) = (x,0,t - \tfrac{xy}{2}). \end{displaymath}
With this notation, we prove the following variant of the Loomis-Whitney inequality:
\begin{thm}\label{mainIntro2} Let $K \subset \R^{3}$ (or $K \subset \He$) be Lebesgue measurable. Then
\begin{equation}\label{form18} |K| \lesssim |\pi_{1}(K)|^{2/3} \cdot |\pi_{2}(K)|^{2/3}. \end{equation}
\end{thm}
Theorem \ref{mainIntro2} will be derived as a corollary of Theorem \ref{mainIntro}. It is easy to see that the exponents in \eqref{form18} are sharp by considering  rectangles of the form $[-r,r] \times [-r,r] \times [-r^{2},r^{2}]$. Besides the difference in the definition of projections, there is another obvious difference between (the case $n = 3$ of) the standard Loomis-Whitney inequality \eqref{LWIneq}, and \eqref{form18}: the former bounds the volume of $K$ in terms of three projections, and the latter in terms of only two projections. One might therefore ask: is there a version of \eqref{LWIneq} for two orthogonal projections $\R^{3} \to \R^{2}$ -- and does it look like \eqref{form18}? The answer is negative. This is a very special case of \cite[Theorem 1.13]{MR2377493}, but perhaps it is illustrative to see an explicit computation:
\begin{ex} Consider the two coordinate projections $\tilde{\pi}_{1},\tilde{\pi}_{2}$ in $\R^{3}$ to the $xt$- and $yt$-planes. If $K = [0,1]^{2} \times [0,\delta]$, then $|K| = \delta$, and also $|\tilde{\pi}_{1}(K)| = \delta = |\tilde{\pi}_{2}(K)|$. So, for $\delta > 0$ small, an inequality of the form
\begin{equation}\label{form19} |K| \lesssim |\tilde{\pi}_{1}(K)|^{\lambda} \cdot |\tilde{\pi}_{2}(K)|^{\lambda} \end{equation}
can only hold for $\lambda \leq \tfrac{1}{2}$. On the other hand, if $K_{R} = [0,R]^{3}$, with $R \gg 1$, then $|K_{R}| = R^{3}$ and $|\tilde{\pi}_{1}(K_{R})| = R^{2} = |\tilde{\pi}_{2}(K_{R})|$, so \eqref{form19} can only hold for $\lambda \geq \tfrac{3}{4}$. The latter example naturally does not contradict \eqref{form18}: note that $|\pi_{j}(K_{R})| \sim R^{3}$ for $R \gg 1$. \end{ex}

We also mention that Theorem \ref{mainIntro2} is related to \emph{Brascamp-Lieb inequalities}, but, to the best of our knowledge, does not follow from existing results. We direct the reader to e.g. \cite{2018arXiv181111052B,MR2377493,MR412366} and the references therein.

In $\R^{n}$, it is well-known that the Loomis-Whitney inequality implies the \emph{Gagliardo-Nirenberg-Sobolev inequality}
\begin{equation}\label{GSN} \|f\|_{n/(n - 1)} \leq \prod_{j = 1}^{n} \|\partial_{j}f\|_{1}^{1/n}, \qquad f \in C^{1}_{c}(\R^{n}). \end{equation}
Similarly, we obtain an $\He$-analogue of \eqref{GSN} as a corollary of Theorem \ref{mainIntro2}:
\begin{thm}\label{mainIntro3} Let $f \in BV(\He)$. Then,
\begin{equation}\label{GSNHe} \|f\|_{4/3} \lesssim \sqrt{\|Xf\|\|Yf\|}. \end{equation}
\end{thm}
Here
\begin{displaymath} X = \partial_{x} - \tfrac{y}{2}\partial_{t} \quad \text{and} \quad Y = \partial_{y} + \tfrac{x}{2}\partial_{t} \end{displaymath}
are the standard left-invariant "horizontal" vector fields in $\He$, and $BV(\He)$ refers to functions $f \in L^{1}(\He)$ whose distributional $X$ and $Y$ derivatives are signed Radon measures with finite total variation, denoted $\|\cdot\|$. Theorem \ref{mainIntro3} presents a sharper version of the well-known "geometric" Sobolev inequality
\begin{equation}\label{sobolev} \|f\|_{4/3} \lesssim \|\nabla_{\He}f\|, \qquad f \in BV(\He), \end{equation}
proven by Pansu \cite{MR676380} as a corollary of the
isoperimetric inequality in $\He$. Here $\nabla_{\He}f = (Xf,Yf)$.  Versions of geometric Sobolev inequalities and isoperimetric inequalities were obtained in a more general framework by several authors, for instance in \cite{MR1312686,MR1404326}.
A proof of
\eqref{sobolev}, using the fundamental solution of the sub-Laplace
operator $\bigtriangleup_{\He}$, is discussed in \cite[Section
5.3]{MR2312336}, following the approach of \cite{MR1312686}.
On the other hand, since Theorem \ref{mainIntro3}
is derived from Theorem \ref{mainIntro2}, which in turn is a
corollary of Theorem \ref{mainIntro}, our proof of the inequality
\eqref{GSNHe}, and hence \eqref{sobolev}, uses nothing but plane
geometry!

It seems plausible that a version of Theorem \ref{mainIntro2} could also hold in higher dimensional Heisenberg groups, but we are not currently able to prove it:
\begin{question} Let $K \subset \He^{n} \cong \R^{2n + 1}$ be Lebesgue measurable, and let $\pi_{1},\ldots,\pi_{2n}$ be the vertical projections to the planes perpendicular to the $2n$ standard unit vectors $e_{j} \in \R^{2n} \times \{0\} \subset \R^{2n + 1}$. Then,
\begin{displaymath} |K| \lesssim \prod_{j = 1}^{2n} |\pi_{j}(K)|^{(n + 1)/[n(2n + 1)]}. \end{displaymath}
\end{question}


\section{An incidence estimate in the plane}


\begin{definition}[A metric on lines] Let $\mathcal{Q}_{0}$ be the set of lines in $\R^{2}$ whose slope does not exceed $45^{\circ}$, and which intersect the $y$-axis in $\{0\} \times [-1,1]$:
\begin{displaymath} \mathcal{Q}_{0} := \{\ell_{(a,b)} := \{(x,y)\in \mathbb{R}^2\colon y = ax + b\} : |a|,|b| \leq 1\}. \end{displaymath}
For $\ell_{(a,b)},\ell_{(c,d)} \in \mathcal{Q}_{0}$, write
\begin{displaymath} d(\ell_{(a,b)},\ell_{(c,d)}) := |(a,b) - (c,d)|. \end{displaymath}
\end{definition}
Let $0 < \delta \leq 1$. A point $p \in \R^{2}$ is
\emph{$\delta$-incident} to a line $\ell \subset \R^{2}$ if $p\in
\ell(\delta)$. The parameter $\delta
> 0$ will be fixed in this section, and the $\delta$-incidence of
$p$ and $\ell$ will be denoted $p \sim \ell$. For another parameter $\epsilon \in [\delta,1]$, we say that two lines $\ell_{1},\ell_{2} \in
\mathcal{Q}_{0}$ are called \emph{$\epsilon$-separated} if
$d(\ell_{1},\ell_{2}) \geq \epsilon$. The point here is that we will only ever consider $\delta$-incidences between points and lines, but sometimes the results can be improved by assuming that the lines are $\epsilon$-separated, and not just $\delta$-separated.

We record a fairly obvious lemma:
\begin{lemma}\label{l:star} Let $0 < \delta \leq \epsilon < 1$. Let $\mathcal{L} \subset \mathcal{Q}_{0}$ be an $\epsilon$-separated family of lines, all $\delta$-incident to a common point $p \in Q_{0} := [-1,1]^{2}$. Then $|\mathcal{L}| \leq A\epsilon^{-1}$, where $A \geq 1$ is an absolute constant. \end{lemma}
\begin{proof} Write $p = (x_{0},y_{0}) \in Q_{0}$. For every $\ell = \ell_{(a,b)} \in \mathcal{L}$, there exists $x_{\ell} \in \R$ such that $|x_{\ell} - x_{0}| + |y_{0} - (ax_{\ell} + b)| \lesssim \delta$. Since $|a| \leq 1$, it follows that $|b - (-x_{0}a + y_{0})| \lesssim \delta$
and hence $(a,b)$ also lies at distance $\lesssim \delta$ from the
line $\{y = -x_{0}x + y_{0}\}$. Noting that $\epsilon \geq \delta$, there can be at most $\lesssim
\epsilon^{-1}$ such $\epsilon$-separated choices of $(a,b)$, as
claimed.  \end{proof} Our restriction to the lines in
$\mathcal{Q}_{0}$ is purely a matter of convenience; it allows us
to define the metric $d$ in a neat way, which (i) corresponds to
the "geometric intuition" of what the $\delta$-separation of lines
should mean, and (ii) behaves well under point-line-duality.

For a (finite) set $P \subset \R^{2}$, and a (finite) family of
lines $\mathcal{L}$ in $\R^{2}$, we write
\begin{displaymath} \mathcal{I}(P,\mathcal{L}) := \{(p,\ell) : p \sim \ell\} = \{(p,\ell) : p \in \ell(\delta)\}.  \end{displaymath}
Here is the main result of this section:
\begin{thm}\label{mainIncidence} Let $P \subset Q_{0} := [-1,1]^{2}$ be a $\delta$-separated set, and let $\mathcal{L} \subset \mathcal{Q}_{0}$ be a $\delta$-separated family of lines. Then,
\begin{displaymath} |\mathcal{I}(P,\mathcal{L})| \lesssim |P|^{2/3}|\mathcal{L}|^{2/3} \cdot \delta^{-1/3}. \end{displaymath}
\end{thm}
Theorem \ref{mainIncidence} will be derived as a corollary of the
following reformulation, which bounds the number of \emph{$k$-rich
points} for a given $\epsilon$-separated line family $\mathcal{L}$
in $\R^{2}$. For Theorem \ref{mainIncidence}, we will only need the case $\epsilon = \delta$, but proving the more general statement presents no additional challenges. Given a line family $\mathcal{L}$, and an integer $k \geq 1$,
a point $p \in \R^{2}$ is is called \emph{$k$-rich} (relative to $\mathcal{L}$) if $p \sim
\ell$ for $\geq k$ distinct $\ell \in \mathcal{L}$.
\begin{thm}\label{mainRich} Let $0 < \delta \leq \epsilon \leq 1$. Let $\mathcal{L} \subset \mathcal{Q}_{0}$ be an $\epsilon$-separated family of lines, and let $P \subset Q_{0}$ be a $\delta$-separated set of $k$-rich points (relative to $\mathcal{L}$) with $k \geq 2$. Then,
\begin{equation}\label{eq:rich} |P| \lesssim \frac{|\mathcal{L}|^{2}}{k^{3}} \cdot \epsilon^{-1}. \end{equation}
\end{thm}
\begin{remark}\label{rem1} First, we mention that the definition of "$p \sim \ell$" could also be modified by requiring that $p \in \ell(C\delta)$, where $C \geq 1$ is a fixed constant. Then both Theorems \ref{mainIncidence} and \ref{mainRich} would continue to hold, with the same proofs but with a worse constant, depending only on $C$.

Second, the bound \eqref{eq:rich} should be compared with the next
classical estimate, which follows from the Szemer\'edi-Trotter
incidence theorem \cite{MR729791}: given a family of lines $\mathcal{L}$ in
$\R^{2}$, the set of points in $\R^{2}$ contained on $\geq k$
lines has cardinality
\begin{equation}\label{form10} \lesssim \frac{|\mathcal{L}|^{2}}{k^{3}} + \frac{|\mathcal{L}|}{k}. \end{equation}
It seems suspicious that \eqref{eq:rich} is completely missing the
second term in \eqref{form10}, which is indeed necessary: think of "$k$-stars", where $|\mathcal{L}|/k$
points, each, lie on $k$ lines in $\mathcal{L}$. Since such a
construction is possible in the context of Theorem \ref{mainRich},
it has to be the case that
\begin{equation}\label{form11} \frac{|\mathcal{L}|}{k} \lesssim \frac{|\mathcal{L}|^{2}}{k^{3}} \cdot \epsilon^{-1}. \end{equation}
This is true: since the lines in $\mathcal{L}$ are
$\epsilon$-separated, and $\epsilon \geq \delta$, no point in $Q_{0}$ can be $\delta$-incident
to more than $\min\{|\mathcal{L}|,A\epsilon^{-1}\}$ lines in
$\mathcal{L}$. Thus, we may assume in proving Theorem
\ref{mainRich} that $k \leq \min\{|\mathcal{L}|,A\epsilon^{-1}\}
\lesssim \sqrt{|\mathcal{L}|\epsilon^{-1}}$. This bound is
equivalent to \eqref{form11}.

Third, the bound \eqref{eq:rich} should be compared with the recent work of Guth, Solomon, and Wang \cite[Theorem 1.1]{MR4034922}. Under the hypotheses and terminology of Theorem \ref{mainRich}, the authors in \cite{MR4034922} prove that the number of $\delta$-separated $k$-rich points relative to $\mathcal{L}$ is $\lessapprox |\mathcal{L}|^{2}/k^{3}$ if \emph{a priori} $k \gg \delta \epsilon^{-2}$. Thus, for $0 < \epsilon \ll 1$, the upper bound in \cite{MR4034922} is much stronger than \eqref{eq:rich}, but it is only applicable for sufficiently large values of $k$. To prove Theorem \ref{mainIncidence}, we also need information about small values of $k$. For the case $\epsilon = \delta$ in particular, \cite[Theorem 1.1]{MR4034922} does not seem to contain any information, since if $k \gg \delta \epsilon^{-2} = \delta^{-1}$, the set of $k$-rich points is always empty by Lemma \ref{l:star}.  \end{remark}

In the proof of Theorem \ref{mainRich}, we will employ the
following polynomial cell decomposition lemma of Guth and Katz
\cite[Theorem 4.1]{MR3272924}:
\begin{lemma}\label{l:GK} Let $P \subset \R^{2}$ be a finite set, and let $D \geq 1$ be an integer. Then, there exists a polynomial $p \colon \R^{2} \to \R$ of degree $\deg p \leq D$ such that the following holds. Writing
\begin{displaymath} Z := \{x \in \R^{2} : p(x) = 0\}, \end{displaymath}
the complement $\R^{2} \, \setminus \, Z$ is the union of
$\lesssim D^{2}$ open cells $O$ such that $|O \cap P| \lesssim
|P|/D^{2}$. \end{lemma}

We are then ready to prove Theorem \ref{mainRich}:

\begin{proof}[Proof of Theorem \ref{mainRich}] While proving Theorem \ref{mainRich}, we may assume that
\begin{equation}\label{form5} 2 \leq k \leq \min\{|\mathcal{L}|,A\epsilon^{-1}\}, \end{equation}
where $A$ is the constant from Lemma \ref{l:star}, see Remark
\ref{rem1}. In addition to \eqref{form5}, we may also assume that either (a) $k \geq A_{0}$, or (b) $\epsilon \geq A_{0}\delta$, where $A_{0} \geq 1$ is an absolute constant of our choosing. Indeed, if both $\epsilon \leq A_{0}\delta$ and $k \leq A_{0}$, the right hand side of \eqref{eq:rich} is $\geq A_{0}^{-4} |\mathcal{L}|^{2} \cdot \delta^{-1} \gtrsim |\mathcal{L}| \cdot \delta^{-1}$. But clearly the number of $k$-rich points is no larger than the number of $1$-rich points, which is $\lesssim |\mathcal{L}| \cdot \delta^{-1}$, recalling that $P$ is $\delta$-separated. In both cases (a) and (b) we can infer the following geometric observation, which will be useful later in the argument:
\begin{lemma}\label{l:angular} The following holds if $A_{0} \geq 1$ is large enough, and either (a) $\epsilon \geq \delta$ and $k \geq A_{0}$ or (b) $\epsilon \geq A_{0}\delta$ and $k \geq 2$. If $\mathcal{L}(p) \subset \mathcal{L}$ is an $\epsilon$-separated set of lines which are all $\delta$-incident to a common point $p \in Q_{0}$, with $N := |\mathcal{L}(p)| \geq k$, then there are subsets $\mathcal{L}_{1}(p),\mathcal{L}_{2}(p) \subset \mathcal{L}(p)$ of cardinalities $|\mathcal{L}_{1}(p)| \sim N \sim |\mathcal{L}_{2}(p)|$ such that $\angle(\ell_{1},\ell_{2}) \gtrsim N\epsilon$ for all $(\ell_{1},\ell_{2}) \in \mathcal{L}_{1}(p) \times \mathcal{L}_{2}(p)$. \end{lemma}

We omit the easy proof. We will assume that either (a) or (b) holds, so we have the conclusion of Lemma \ref{l:angular}. To prove Theorem \ref{mainRich}, we fix $k$ as in \eqref{form5} (and also with $k \geq A_{0}$ in case (a) holds),
and make a counter assumption: there exists an $\epsilon$-separated
line family $\mathcal{L}$, and a $\delta$-separated set $P \subset
Q_{0}$ of cardinality
\begin{equation}\label{form6} |P| \geq C\frac{|\mathcal{L}|^{2}}{k^{3}} \cdot \epsilon^{-1} \end{equation}
such that every point in $P$ is $\geq k$-rich (relative to
$\mathcal{L}$). Here $C \geq 1$ is some large absolute constant to be
determined later. We apply the cell decomposition lemma, Lemma
\ref{l:GK}, with $D := \floor{C_{\deg}|\mathcal{L}|/k} \geq 1$, where $C_{\deg} \geq 1$ is another constant satisfying
\begin{displaymath} 1 \ll C_{\deg} \ll \sqrt{C}. \end{displaymath}
The precise requirements will become clear during the proof. We obtain
a polynomial $p \colon \R^{2} \to \R$ of degree $\deg p \leq
C_{\deg}|\mathcal{L}|/k$, and a collection of "cells" $\mathcal{O}$, with
$|\mathcal{O}| \lesssim |\mathcal{L}|^{2}/k^{2}$, such that
\begin{equation}\label{form12} |O \cap P| \lesssim \frac{|P|}{D^{2}} \sim \frac{k^{2}|P|}{C_{\deg}^{2}|\mathcal{L}|^{2}}, \qquad O \in \mathcal{O}. \end{equation}
We split the set $P$ into two parts: the points "well inside" the
cells $O \in \mathcal{O}$, and the part "close" to $Z = \{p =
0\}$. The plan is to show that both parts have cardinality $<
|P|/2$, which gives a contradiction, and completes the proof.
Precisely, we write
\begin{displaymath} O' := O \, \setminus \, Z(\delta), \qquad O \in \mathcal{O}. \end{displaymath}
We then write $P = P_{1} \cup P_{2}$, where
\begin{displaymath} P_{1} := P \cap \bigcup_{O \in \mathcal{O}} O' \quad \text{and} \quad P_{2} := P \cap Z(\delta). \end{displaymath}
\subsubsection{Proof that $|P_{1}| < |P|/2$} We start with the following observation: if a line $\ell \subset \R^{2}$ (from $\mathcal{L}$ if desired) is $\delta$-incident to a point in $p \in O'$, then $\ell \cap O \neq \emptyset$. It follows: \emph{every line $\ell \subset \R^{2}$ is $\delta$-incident to a point in at most $\deg p + 1$ sets $O'$.} Indeed, if $\ell$ violated this, then it would intersect $> \deg p + 1$ distinct cells $O \in \mathcal{O}$, and hence cross $Z$ in $\geq \deg p + 1$ distinct points. By B\'ezout's theorem, this would force $\ell \subset Z$, and hence $\ell(\delta) \subset Z(\delta)$. In particular, $\ell$ could not be $\delta$-incident to any points in any of the cells $O' \subset \R^{2} \, \setminus \, Z(\delta)$. We learned this argument from \cite[Lemma 3.2]{MR3454378}.

We infer the following useful corollary of the previous
observation.  For $O \in \mathcal{O}$ fixed, we write
$\mathcal{L}_{O}$ for the subset of $\mathcal{L}$ which are
$\delta$-incident to at least one point in $O'$. Then,
\begin{equation}\label{form13} \sum_{O \in \mathcal{O}} |\mathcal{L}_{O}| = \sum_{\ell \in \mathcal{L}} |\{O \in \mathcal{O} : \ell \in \mathcal{L}_{O}\}| \leq [\deg p + 1] \cdot |\mathcal{L}|. \end{equation}
We are now ready to estimate the number of points in $P_{1}$. We
will first use the $k$-richness of the points in $P_{1} \subset
P$, and then the trivial bound $|\mathcal{I}(P',\mathcal{L}')|
\leq |P'||\mathcal{L}'|$:
\begin{align*} k|P_{1}| \leq |\mathcal{I}(P_{1},\mathcal{L})| & = \sum_{O \in \mathcal{O}} |\mathcal{I}(P \cap O',\mathcal{L}_{O})|\\
& \stackrel{\eqref{form12}}{\lesssim} \frac{k^{2}|P|}{C_{\deg}^{2}|\mathcal{L}|^{2}} \sum_{O \in \mathcal{O}} |\mathcal{L}_{O}|\\
& \stackrel{\eqref{form13}}{\leq}
\frac{k^{2}|P|}{C_{\deg}^{2}|\mathcal{L}|^{2}} \cdot [\deg p + 1] \cdot
|\mathcal{L}| \lesssim \frac{k|P|}{C_{\deg}}, \end{align*} recalling that $\deg p
\leq C_{\deg}|\mathcal{L}|/k$. If $C_{\deg} \geq 1$ was chosen large enough, this
shows that $|P_{1}| < |P|/2$, as desired.

\subsubsection{Proof that $|P_{2}| < |P|/2$} The argument here follows rather closely \cite[\S 5.2]{MR3498792}. There are certain troublesome points $P_{2,\mathrm{bad}} \subset P_{2}$ whose cardinality we bound first: they are the points $p \in P_{2}$ such that $B(p,2\delta)$ contains a component of $Z$. By Harnack's curve theorem \cite{MR1509883}, the number of components $Z_{1},\ldots,Z_{N}$ of $Z$ is bounded by
\begin{displaymath} N \lesssim [\deg p]^{2} \lesssim \frac{C_{\deg}^{2}|\mathcal{L}|^{2}}{k^{2}} \leq A\frac{C_{\deg}^{2}|\mathcal{L}|^{2}}{k^{3}} \cdot \epsilon^{-1} \leq A\frac{C_{\deg}^{2}|P|}{C}, \end{displaymath}
recalling from \eqref{form5} that $k \leq A\epsilon^{-1}$, and then
applying the counter assumption \eqref{form6}. Since the points in
$P$ are $\delta$-separated, and the $2\delta$-neighbourhood of
every point in $P_{2,\mathrm{bad}}$ contains one of the $N$
components of $Z$, we conclude that $|P_{2,\mathrm{bad}}| \lesssim
N \lesssim C_{\deg}^{2}|P|/C$. Choosing $C \geq 1$ large enough, and then $C_{\deg} \ll \sqrt{C}$, we find
that $|P_{2,\mathrm{bad}}| < |P|/4$. To conclude the proof, it
remains to prove that $|P_{2,\mathrm{good}}| < |P|/4$, where
$P_{2,\mathrm{good}} := P_{2} \, \setminus \, P_{2,\mathrm{bad}}$.
\begin{figure}[h!]
\begin{center}
\includegraphics[scale = 0.9]{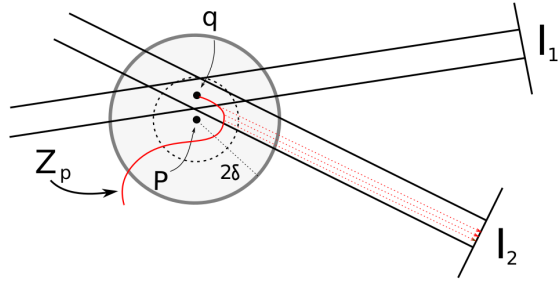}
\caption{$Z \cap B(p,2\delta)$ must have a large projection in one
of two directions with a positive angle.}\label{fig1}
\end{center}
\end{figure}

We make a geometric observation, depicted in Figure \ref{fig1}.
Fix $p \in P_{2,\mathrm{good}}$. Since $p \in Z(\delta)$, and
$B(p,2\delta)$ contains no component of $Z$, we infer that some
component $Z_{p}$ of $Z$ contains a point $q \in B(p,\delta)$, and
also intersects $\R^{2} \, \setminus \, B(p,2\delta)$. We spend a
moment studying the orthogonal projections of $Z_{p} \cap
B(p,2\delta)$ to lines through the origin. Fix two such lines
$L_{1},L_{2}$ with angle $\angle(L_{1},L_{2}) =: \alpha \geq
\delta$, and let $\pi_{j} \colon \R^{2} \to L_{j}$ be the
orthogonal projection. Evidently $\pi_{j}(q) \in \pi_{j}(Z_{p}
\cap B(p,2\delta))$. Let $\Pi_{j}$ be the (possibly degenerate) component interval of
$\pi_{j}(Z_{p} \cap B(p,2\delta))$ containing $\pi_{j}(q)$. We
claim that
\begin{equation}\label{form17} \max\{|\Pi_{1}|,|\Pi_{2}|\} \geq c\alpha\delta, \end{equation}
where $c > 0$ is a suitable absolute constant. Assume to the
contrary that $\max\{|\Pi_{1}|,|\Pi_{2}|\} < c\alpha \delta$. This
implies that there are points $x_{j}^{1},x_{j}^{2} \in L_{1}$ at
distance $< c\alpha \delta$ from $\pi_{j}(q)$ which are not in
$\pi_{j}(Z_{p} \cap B(p,2\delta))$. Write $I_{j} :=
[x_{j}^{1},x_{j}^{2}] \subset L_{j}$, $j \in \{1,2\}$, and note
that
\begin{displaymath} Q := \pi_{1}^{-1}(I_{1}) \cap \pi_{2}^{-1}(I_{2}) \end{displaymath}
is a rectangular box of diameter $\lesssim (c\alpha \delta)/\alpha
= c\delta$. Since $q \in Q \cap B(p,\delta)$, choosing $c > 0$
sufficiently small allows us to conclude that $Q \subset
B(p,2\delta)$. Since $x_{j}^{i} \notin \pi_{j}(Z_{p} \cap
B(p,2\delta))$ for $i,j\in \{1,2\}$, we have
\begin{displaymath} [Z_{p} \cap B(p,2\delta)] \cap \pi_{j}^{-1}(x_{j}^{i}) = \emptyset. \end{displaymath}
The boundary of $Q$ is contained in the union of the (four) lines
$\pi_{j}^{-1}(x_{j}^{i})$, $i,j \in \{1,2\}$, so we infer that
\begin{equation}\label{form16} [Z_{p} \cap B(p,2\delta)] \cap \partial Q = \emptyset. \end{equation}
However, $Z_{p}$ is a connected set meeting both $Q$ (at $q$) and
$\R^{2} \, \setminus \, Q$ (recalling that $Q \subset
B(p,2\delta)$), so $Z_{p} \cap \partial Q \neq \emptyset$. Using again
that $Q \subset B(p,2\delta)$, this violates \eqref{form16}, and
proves \eqref{form17}. Now that \eqref{form17} has been proven, we
can relax it a bit by eliminating the reference to the special
component $Z_{p}$: we have shown that if $p \in
P_{2,\mathrm{good}}$, and $L_{1},L_{2}$ are two lines through the
origin with $\angle(L_{1},L_{2}) = \alpha \geq \delta$, then
\begin{equation}\label{form15} \max\{|\pi_{L_{1}}(Z \cap B(p,2\delta))|,|\pi_{L_{2}}(Z \cap B(p,2\delta))|\} \geq c\alpha \delta. \end{equation}
We apply this as follows: let $\ell_{1},\ell_{2} \in \mathcal{L}$
be two lines $\delta$-incident to $p$ with
$\angle(\ell_{1},\ell_{2}) =: \alpha \geq \delta$, and let
$\pi_{1},\pi_{2}$ be the orthogonal projections to $L_{1} :=
\ell_{1}^{\perp}$ and $L_{2} := \ell_{2}^{\perp}$, respectively.
It follows from \eqref{form15}, and $B(p,2\delta) \subset
\ell_{i}(4\delta)$, that
\begin{equation}\label{form7} \max\{|\pi_{1}([Z \cap B(p,2\delta)] \cap \ell_{1}(4\delta))|, |\pi_{2}([Z \cap B(p,2\delta)] \cap \ell_{2}(4\delta))|\} \geq c\alpha\delta. \end{equation}
To exploit this information, recall that $p \in
P_{2,\mathrm{good}} \subset P$ is $\geq k$-rich, with $k \geq 2$, and the lines in $\mathcal{L}$ are $\epsilon$-separated, with $\epsilon \geq \delta$. So, using Lemma \ref{l:angular}, we may
isolate two collections $\mathcal{L}_{1}(p)$ and
$\mathcal{L}_{2}(p)$ of $\gtrsim k$ lines in $\mathcal{L}$, all
$\delta$-incident to $p$, such that $\angle(\ell_{1},\ell_{2})
\gtrsim k\epsilon$ for all pairs $(\ell_{1},\ell_{2}) \in
\mathcal{L}_{1}(p) \times \mathcal{L}_{2}(p)$. By \eqref{form7},
the following holds for either $\mathcal{L}_{1}(p)$ or
$\mathcal{L}_{2}(p)$:
\begin{equation}\label{form8} |\pi_{\ell^{\perp}}([Z \cap B(p,2\delta) ] \cap \ell(4\delta))| \geq ck\epsilon\delta, \qquad \ell \in \mathcal{L}_{j}(p), \end{equation}
where $c > 0$ might be a bit smaller than in \eqref{form7}. Motivated by this observation, we say that $(p,\ell) \in
P_{2,\mathrm{good}} \times \mathcal{L}$ is a \emph{good incidence}
if $\ell \sim p$, and \eqref{form8} holds. Since
$|\mathcal{L}_{j}(p)| \gtrsim k$ for all $p \in
P_{2,\mathrm{good}}$ and $j\in \{1,2\}$, we see that
\begin{align*} \sum_{\ell \in \mathcal{L}} & |\{p \in P_{2,\mathrm{good}} : (p,\ell) \text{ is a good incidence}\}|\\
& = \sum_{p \in P_{2,\mathrm{good}}} |\{\ell \in \mathcal{L} :
(p,\ell) \text{ is a good incidence}\}| \gtrsim |P_{2,\mathrm{good}}|
\cdot k. \end{align*} Averaging over $\ell \in \mathcal{L}$, we
find a line $\ell_{0} \in \mathcal{L}$ such that
\begin{equation}\label{form9} |\{p \in P_{2,\mathrm{good}} : (p,\ell_0) \text{ is a good incidence}\}| \gtrsim \frac{|P_{2,\mathrm{good}}| \cdot k}{|\mathcal{L}|}. \end{equation}
We will now conclude the proof by inferring, from \eqref{form9},
that (Lebesgue) positively many lines inside $\ell_{0}(4\delta)$
are contained in $Z$. Since however $Z$ has null measure (and
\emph{a fortiori} can contain at most $\deg p$ distinct parallel
lines), we will have reached a contradiction. So, consider a
random line $\ell' \subset \ell_{0}(4\delta)$; more precisely, let
$\tn$ be the uniform distribution on
$\pi_{\ell_{0}^{\perp}}(\ell_{0}(4\delta)) \cong [0,8\delta]$, and
pick
\begin{displaymath} \ell' = \pi_{\ell_{0}^{\perp}}^{-1}\{t\} \subset \ell(4\delta) \end{displaymath}
according to $t \sim \tn$. Whenever $(p,\ell_0)$ is a good
incidence, \eqref{form8} implies that the probability of $\ell'$
hitting $Z \cap \bar{B}(p,2\delta)$ is $\gtrsim k\epsilon$. The
balls $\bar{B}(p,2\delta)$ have bounded overlap as $p$ varies, so
\eqref{form9} implies that the expected number "$\mathbb{E}$" of
intersections between $\ell' \subset \ell_0(4\delta)$ and $Z$ is
\begin{displaymath} \mathbb{E} \gtrsim \frac{|P_{2,\mathrm{good}}| \cdot k^{2}\epsilon}{|\mathcal{L}|}. \end{displaymath}
On the other hand, $\mathbb{E} \leq \deg p \leq
C_{\deg}|\mathcal{L}|/k$: any line with $> \deg p$ intersections with $Z$
is contained in $Z$ by B\'ezout's theorem, and this cannot happen
for a set of lines with positive probability (or even strictly
more than $\deg p$ choices of $\ell'$). So, we infer that
\begin{displaymath} \frac{|P_{2,\mathrm{good}}| \cdot k^{2}\epsilon}{|\mathcal{L}|} \lesssim \mathbb{E} \leq \frac{C_{\deg}|\mathcal{L}|}{k}, \end{displaymath}
which can be rearranged to
\begin{displaymath} |P_{2,\mathrm{good}}| \lesssim \frac{C_{\deg}|\mathcal{L}|^{2}}{k^{3}} \cdot \epsilon^{-1} \leq \frac{C_{\deg}|P|}{C} \leq \frac{|P|}{\sqrt{C}}, \end{displaymath}
using the counter assumption \eqref{form6} in the end, and also recalling that $C_{\deg} \leq \sqrt{C}$. Choosing $C
\geq 1$ large enough, we infer that $|P_{2,\mathrm{good}}| < |P|/4$,
as desired. Since we have now shown that
\begin{displaymath} |P| \leq |P_{1}| + |P_{2,\mathrm{bad}}| + |P_{2,\mathrm{good}}| < \frac{|P|}{2} + \frac{|P|}{4} + \frac{|P|}{4} < |P|, \end{displaymath}
a contradiction (starting from \eqref{form6}) has been reached,
and the proof of Theorem \ref{mainRich} is complete. \end{proof}
We then quickly derive Theorem \ref{mainIncidence}:
\begin{proof}[Proof of Theorem \ref{mainIncidence}] Fix a $\delta$-separated set $P \subset Q_{0}$, and a $\delta$-separated set of lines $\mathcal{L} \subset \mathcal{Q}_{0}$. There is no loss of generality assuming that $|\mathcal{L}| \geq |P|$: if this fails to begin with, one may apply \emph{point-line duality}
to exchange the roles of $P$ and $\mathcal{L}$ and obtain a new
set of points, $P_{\mathcal{L}}$, and a new family of lines,
$\mathcal{L}_{P}$. This is a standard trick, so we only sketch the
details: one associates to every $(a,b) \in P$ the line
$\tilde{\ell}_{(a,b)} = \{y = -ax + b\}$, and to every line
$\ell_{(c,d)} = \{y = cx + d\} \in \mathcal{L}$ the point $(c,d)
\in \R^{2}$. Then, it is clear that $(a,b)$ lies on $\ell_{(c,d)}$
if and only if $(c,d)$ lies on $\tilde{\ell}_{(a,b)}$. Also, the
$\delta$-separated set $P \subset Q_{0}$ gets mapped to a
$\delta$-separated set of lines in $\mathcal{Q}_{0}$, and vice
versa, by our definition of "$\delta$-separation". With a little
work, one can also check that if $(a,b)$ is $\delta$-incident to
$\ell_{(c,d)}$, then $(c,d)$ is $C\delta$-incident to
$\tilde{\ell}_{(a,b)}$. Therefore, with suitable choices of
constants in the definitions, one has
$|\mathcal{I}(P,\mathcal{L})| \lesssim
|\mathcal{I}(P_{\mathcal{L}},\mathcal{L}_{P})|$. But if
$|\mathcal{L}| < |P|$, then $|P_{\mathcal{L}}| <
|\mathcal{L}_{P}|$, and we have arrived at a situation where the
number of lines exceeds the number of points, as desired.

So, we assume that $|\mathcal{L}| \geq |P|$, and in particular
$[|\mathcal{L}|^{2/3}/|P|^{1/3}] \cdot \delta^{-1/3} \geq 1$. For
$j \geq 1$, let
\begin{displaymath} P_{j} := \{p \in P : p \text{ is $k$-rich for some $2^{j - 1} \leq k < 2^{j}$}\}. \end{displaymath}
The set $P_{1}$ consists of the $1$-rich points in $P$, and for these we apply the trivial bound $|\mathcal{I}(P_{1},\mathcal{L})| \leq |P|$. For $j \geq 2$, we apply Theorem \ref{mainRich} as follows:
\begin{align*} |\mathcal{I}(P,\mathcal{L})| & \lesssim |P| + \sum_{j \geq 2} 2^{j} |P_{j}| \lesssim  \sum_{2^{j} \leq [|\mathcal{L}|^{2/3}/|P|^{1/3}] \cdot \delta^{-1/3}} 2^{j}|P|\\
& + \sum_{2^{j} > [|\mathcal{L}|^{2/3}/|P|^{1/3}] \cdot
\delta^{-1/3}} 2^{j} \cdot \frac{|\mathcal{L}|^{2}}{2^{3j}} \cdot
\delta^{-1}.\end{align*} One readily verifies that both sums
above are comparable to $|P|^{2/3}|\mathcal{L}|^{2/3} \cdot
\delta^{-1/3}$, and also $|P| \leq |P|^{2/3}|\mathcal{L}|^{2/3} \cdot
\delta^{-1/3}$ since we assumed $|P| \leq |\mathcal{L}$|. This concludes the proof. \end{proof}

\section{Loomis-Whitney inequality in the Heisenberg
group}\label{s:LW} In this section, we deduce the Loomis-Whitney
inequality in Theorem \ref{mainIntro2} from the planar incidence
bound established in the previous section. We begin by introducing
the Heisenberg concepts and notation carefully. The \emph{first
Heisenberg group} $\mathbb{H}$ is the group $(\mathbb{R}^3,\cdot)$
with the group product
\begin{equation}\label{eq:GroupProd} (x,y,t) \cdot (x',y',t') := (x + x', y + y', t + t' + \tfrac{1}{2}(xy' - yx')). \end{equation}
The \emph{Heisenberg dilation} $\delta_{\lambda}$ with constant $\lambda>0$ is the group isomorphism
\begin{displaymath}
\delta_{\lambda}:\mathbb{H}\to\mathbb{H},\quad \delta_{\lambda}(x,y,t)=(\lambda x,\lambda y,\lambda^2 t).
\end{displaymath}
In geometric measure theory of the sub-Riemannian Heisenberg group \cite{MR3587666}, an important role is played by
\emph{Heisenberg projections} that are adapted to the group and dilation structure of $\mathbb{H}$ and that map onto homogeneous subgroups of $\mathbb{H}$. In the present paper, we only consider two projections associated to two "coordinate" planes introduced below.

Let $\W_{x} := \{(x,0,t) : (x,t) \in \R^{2}\} \subset \He$ and
$\W_{y} = \{(0,y,t) : (y,t) \in \R^{2}\} \subset \He$ be the
\emph{vertical subgroups} of $\He$ containing the $x$-axis and $y$-axis,
respectively. Write also $\mathbb{L}_{x} := \{(x,0,0) : x \in
\R\}$ and $\mathbb{L}_{y} := \{(0,y,0) : y \in \R\}$, so
\begin{itemize}
\item $\mathbb{L}_{x}$ is a complementary \emph{horizontal subgroup} of
$\mathbb{W}_{y}$, and \item $\mathbb{L}_{y}$ is a complementary
\emph{horizontal subgroup} of $\W_{x}$.
\end{itemize}
This means, for example, that every point $p \in \He$ has a unique
decomposition $p = w_{x} \cdot l_{y}$, where $w_{x} \in \W_{x}$
and $l_{y} \in \mathbb{L}_{y}$. Similarly, there is also a unique
decomposition $p = w_{y} \cdot l_{x}$, where $w_{y} \in
\mathbb{W}_{y}$ and $l_{x} \in \mathbb{L}_{x}$. These
decompositions give rise to the \emph{vertical projections}
\begin{displaymath} p \mapsto w_{x} =: \pi_{x}(p) \in \W_{x} \quad \text{and} \quad p \mapsto w_{y} =: \pi_{y}(p) \in \W_{y}. \end{displaymath}
It is immediate from the definition that the fibres of the
projections $\pi_{x}$ and $\pi_{y}$ left cosets of $\mathbb{L}_{y}$ and $\mathbb{L}_{x}$, respectively:
\begin{displaymath} \pi_{x}^{-1}\{w\} = w \cdot \mathbb{L}_{y} \quad \text{and} \quad \pi_{x}^{-1}\{w\} = w
\cdot \mathbb{L}_{x}. \end{displaymath}
Using the group product in \eqref{eq:GroupProd}, it is
also easy to write down explicit expressions for $\pi_{x}$ and
$\pi_{y}$:
\begin{displaymath} \pi_{y}(x,y,t) = (0,y,t + \tfrac{xy}{2}) \quad \text{and} \quad \pi_{x}(x,y,t) = (x,0,t - \tfrac{xy}{2}). \end{displaymath}
If the reader is not comfortable with the Heisenberg group, he can simply identify both $\W_{x}$ and $\W_{y}$ with $\R^{2}$, and consider the maps $(x,y,t) \mapsto (y,t + (xy)/2))$ and $(x,y,t) \mapsto (x,t - (xy)/2)$ without paying attention to their origin. It is clear that $\pi_{x}$ and $\pi_{y}$ are smooth, and hence
locally Lipschitz with respect to the Euclidean metric in
$\R^{3}$. The vertical projections are, in fact, \textbf{not}
Lipschitz with respect to the \emph{Kor\'anyi distance} $d(p,q) =
\|q^{-1} \cdot p\|$, but all the metric concepts which we use in this section (balls, neighborhoods etc.) will be defined using the Euclidean distance.

We recall the statement of Theorem \ref{mainIntro2}:

\begin{thm}\label{main} Let $K \subset \He$ be Lebesgue measurable. Then,
\begin{equation}\label{form1} |K| \lesssim |\pi_{x}(K)|^{2/3} \cdot |\pi_{y}(K)|^{2/3}. \end{equation}
\end{thm}

\begin{remark} On the left hand side of \eqref{form1}, the notation "$|\cdot|$" refers to either Lebesgue measure on $\R^{3}$ or $\mathcal{H}^{4}_{d}$ (which are the same, up to a multiplicative constant). Similarly, on the right hand side of \eqref{form1}, the notation "$|\cdot|$" can either refer to Lebesgue measure on $\R^{2}$, or $\calH^{3}_{d}$ restricted to a vertical subgroup; these measures, again, coincide up to a constant. Below, the notation "$|\cdot|$" may also refer to cardinality, but the meaning should always be clear from the context.
\end{remark}

\subsection{Reduction to a planar incidence problem}
We start with a few geometric observations which will be used in the proof of Theorem \ref{main}.  Fix a ``scale'' parameter $0<\delta < 1$. We write $Q_{0} := [-1,1]^{3} \subset \R^{3}$. The first lemma records that the "tubes"
$\pi_{x}^{-1}(B(p_x,\delta))$ and $\pi_{y}^{-1}(B(p_y,\delta))$
are fairly close to Euclidean $\delta$-tubes inside the bounded
set $Q_0$:
\begin{lemma}\label{l:tubes} There is an absolute constant $A_{1} \geq 1$ such that the following holds. Let $w_x\in \W_x$, $w_y\in \W_y$, and write $B_{x} := B(w_{x},\delta) \cap \W_{x}$ and $B_{y} := B(w_{y},\delta) \cap \W_{y}$. Then
\begin{displaymath} \pi_{x}^{-1}(B_{x}) \cap Q_{0} \subset [w_{x} \cdot \mathbb{L}_{y}](A_{1}\delta) \quad \text{and} \quad \pi_{y}^{-1}(B_{y}) \cap Q_{0} \subset [w_{y} \cdot \mathbb{L}_{x}](A_{1}\delta). \end{displaymath}
\end{lemma}
In other words, the intersection of $\pi_{x}^{-1}(B_{x})$ with
$Q_{0}$ is contained in the Euclidean $A_{1}\delta$-neighbourhood
of the horizontal line $p_{x} \cdot \mathbb{L}_{y}$, and analogously if the roles of $x$ and $y$ are reverted.
\begin{cor}\label{cor1} Let $w_x\in \W_x \cap Q_{0}$, $w_y \in \W_y \cap Q_{0}$, and consider $T_{x} =[w_{x} \cdot \mathbb{L}_{y}](A_{1}\delta)$ and $T_{y} =[w_{y} \cdot \mathbb{L}_{x}](A_{1}\delta)$. Then,
\begin{displaymath} |T_{x} \cap T_{y}| \lesssim \delta^{3}. \end{displaymath}
\end{cor}
\begin{proof} The horizontal lines $w_{x} \cdot \mathbb{L}_{y}$ and $w_{y} \cdot \mathbb{L}_{x}$ hit $Q_{0}$, so they are quantitatively non-vertical; their angles with the $t$-axis are uniformly bounded from below. This implies that the intersection $T_{x} \cap T_{y}$ is fairly transversal, and the upper bound follows. \end{proof}

\begin{lemma}\label{l:core_proj} There exists a constant $A \geq 1$ such that the following holds. If $w_x =(a,0,b) \in \W_x$,
then
\begin{displaymath}\ell:= \pi_{y}(w_{x} \cdot \mathbb{L}_{y}) = \{(0,y,ay+b) : y \in \R\} \end{displaymath}
and
\begin{displaymath}
\pi_{y}\left(\pi_{x}^{-1}(B(w_x,\delta))\cap Q_{0} \right) \subseteq \ell(A\delta).
\end{displaymath}
\end{lemma}

\begin{proof}
An easy computation shows for arbitrary $(x,0,t)\in \W_x$ and $y\in \mathbb{R}$ that
\begin{equation}\label{formula}
\pi_{y}((x,0,t) \cdot (0,y,0))= \pi_{y}\left(x,y,t+\tfrac{xy}{2}\right)=\left(0,y,xy + t\right).
\end{equation}
This establishes the first claim with $(x,t)=(a,b)$ and $y\in \mathbb{R}$.

The second part of the lemma follows from Lemma \ref{l:tubes} since vertical projections are locally Lipschitz with respect to the Euclidean metric. Alternatively, one can use again \eqref{formula} and let
$(x,t)$ range in a $\delta$-disk centered at $(a,b)$.
\end{proof}

\begin{proposition}\label{p:fromProjToIncid}
Let  $P_x$ and $P_y$ be  $\delta$-separated sets in $\W_x\cap Q_{0}$ and $\W_y\cap Q_{0}$, respectively. Set
\begin{displaymath}
\mathcal{L}_y:=\{\pi_y(w_{x} \cdot \mathbb{L}_{y}):\; w_x\in
P_x\}.
\end{displaymath}
Then $\mathcal{L}_y$ is a $\delta$-separated set of lines in $\mathcal{Q}_{0}$.
Moreover, if $w_{x} \in P_{x}$, $w_{y} \in P_{y}$, and
\begin{equation}\label{eq:intersection}
\pi_x^{-1}( B(w_x,\delta)) \cap \pi_y^{-1}( B(w_y,\delta))\cap Q_{0}\neq \emptyset
\end{equation}
then $w_y$ is $(1+A)\delta$-incident to $\pi_y(w_{x} \cdot \mathbb{L}_{y})$.
\end{proposition}

\begin{proof}
We first observe that $\mathcal{L}_y$ is a $\delta$-separated set of
lines. Indeed, if $w_x,w_x' \in P_{x}$ are distinct, then Lemma \ref{l:core_proj} shows that
$\pi_y(w_x\cdot \mathbb{L}_y)$ and $\pi_y(w_x'\cdot \mathbb{L}_y)$
are two lines in $\W_{y} \cong \mathbb{R}^2$ of the form
\begin{displaymath}
\ell:=\{(y, a y + b):\,  y\in\mathbb{R}\}\quad \text{and}\quad
\ell':=\{(y, a' y + b'):\,  y\in\mathbb{R}\}
\end{displaymath}
with $d(\ell,\ell')=|(a,b)-(a',b')|\geq \delta$, and $|a|,|a'|,|b|,|b'| \leq 1$ (the latter condition ensures that $\ell,\ell' \in \mathcal{Q}_{0}$, recalling we only defined the metric $d$ on $\mathcal{Q}_{0}$).

Next we assume that \eqref{eq:intersection} holds. Using Lemma \ref{l:core_proj}, this implies that
\begin{displaymath}
[\pi_y(w_x\cdot \mathbb{L}_y)](A\delta) \cap B(w_{y},\delta) \supset \pi_y\left( \pi_x^{-1}( B(w_x,\delta))\cap Q_{0}\right) \cap B(w_y,\delta)\neq \emptyset.
\end{displaymath}
We infer that $w_{y} \in [\pi_{y}(w_{x} \cdot \mathbb{L}_{y})]([1 + A]\delta)$, as claimed. \end{proof}

\begin{proof}[Proof of Theorem
\ref{main}]
First, we may assume that $K$ is compact, by the inner regularity of Lebesgue measure. Then, we may assume
that $K \subset \tfrac{1}{2}Q_{0}$, since both sides of \eqref{form1} scale
in the same way with respect to the Heisenberg dilations $\delta_{r}$. Indeed, since the
Jacobian determinant of $\delta_r$ is $r^4$, we have
$
 |\delta_r K| = r^4 |K|.
$
On the other hand, dilations commute with vertical projections,
and the maps $\delta_{r}|_{\mathbb{W}_x}$ and
$\delta_{r}|_{\mathbb{W}_y}$ have Jacobian determinant $r^3$, so
\begin{displaymath}
|\pi_{x}(\delta_r K)|^{2/3} \cdot |\pi_{y}(\delta_r K)|^{2/3} =
|\delta_r (\pi_{x}(K))|^{2/3} \cdot |\delta_r(\pi_{y}(K))|^{2/3} =
r^4\, |\pi_{x}(K)|^{2/3} \cdot |\pi_{y}(K)|^{2/3}.
\end{displaymath}
Thus, we may and will assume that $K\subset \tfrac{1}{2}Q_{0}$, which implies that $\pi_{x}(K) \subset \W_{x} \cap Q_{0}$ and $\pi_{y}(K) \subset \W_{y} \cap Q_{0}$. Since $\pi_{x}(K)$ and $\pi_{y}(K)$ are bounded, there exist finite maximal $\delta$-separated subsets $P_{x} \subset \pi_{x}(K)$ and $P_{y} \subset \pi_{y}(K)$ for any  "scale" parameter $0 < \delta < 1$. Fix $\varepsilon>0$. Then for all $\delta  > 0$  small enough (depending on $K$ and $\varepsilon$), we have
\begin{equation}\label{eq:parameter_bound} \delta^{2}[\card P_{x}] \lesssim |\pi_{x}(K)|+\varepsilon \quad \text{and} \quad \delta^{2}[\card P_{y}] \lesssim |\pi_{y}(K)|+\varepsilon.\end{equation}
To improve clarity, we exceptionally use the notation "$\card$" for cardinality within this proof. The parameter $\varepsilon$  is used here only to handle the case where $|\pi_{x}(K)|=0$ or $|\pi_{y}(K)|=0$.
Now, it suffices to prove for $\delta$ as in \eqref{eq:parameter_bound} that
\begin{equation}\label{form2} |K| \lesssim [\card P_{x}]^{2/3}[\card P_{y}]^{2/3} \cdot \delta^{8/3}.  \end{equation}
This will yield
\begin{equation*} |K| \lesssim (|\pi_{x}(K)|+\varepsilon)^{2/3} \cdot( |\pi_{y}(K)|+\varepsilon)^{2/3}, \end{equation*}
and the theorem follows by letting $\varepsilon \to 0$.

We will establish \eqref{form2} as a corollary of Theorem \ref{mainIncidence}. In order to relate \eqref{form2} to
a set $\mathcal{I}(P,\mathcal{L})$ of incidences, we first recall that
\begin{displaymath}
\pi_x(K) \subset \bigcup_{w_x \in  P_x} B(w_x,\delta) \quad
\text{and}  \quad \pi_y(K) \subset \bigcup_{w_y \in  P_y}
B(w_y,\delta),
\end{displaymath}
and hence
\begin{align*}
K&\subseteq  \bigcup_{(w_x,w_y)\in P_x\times P_y} \pi_x^{-1}( B(w_x,\delta)) \cap \pi_y^{-1}( B(w_y,\delta))\cap Q_{0}.
\end{align*}
It follows from Lemma \ref{l:tubes} and Corollary \ref{cor1} that
\begin{equation}\label{eq:K_upper_bound}
|K|\lesssim \delta^3 \card \{(w_x,w_y)\in
P_x\times P_y:\,
 \pi_x^{-1}( B(w_x,\delta)) \cap \pi_y^{-1}( B(w_y,\delta))\cap Q_{0}
\neq \emptyset \}.
\end{equation}
To control the cardinality that appears on the right, we use
Proposition
\ref{p:fromProjToIncid}. It allows us to deduce from  \eqref{eq:K_upper_bound}  that
\begin{displaymath}
|K|\lesssim \delta^3\, \card \mathcal{I}_{(1 + A)\delta}(P_y,\mathcal{L}_y) ,
\end{displaymath}
where $\mathcal{I}_{(1 + A)\delta}(P_y,\mathcal{L}_y)$ is the set of $(1+A)\delta$-incidences between the points in $P_y$ and the lines in $\mathcal{L}_y := \{\pi_{y}(w_{x} \cdot \mathbb{L}_{y}) : w_{x} \in P_{x}\} \subset \mathcal{Q}_{0}$.
Since $\card \mathcal{L}_y = \card P_x$, the proof of \eqref{form2}, and hence Theorem \ref{main},  is then reduced to showing
\begin{equation}\label{eq:desiredIncidences}
 \card \mathcal{I}_{(1 + A)\delta}(P_y,\mathcal{L}_y) \lesssim [\card \mathcal{L}_y]^{\frac{2}{3}}[\card P_y]^{\frac{2}{3}} \cdot \delta^{-\frac{1}{3}}.
\end{equation}
But since $P_y$ consists of $\delta$-separated points, and $\mathcal{L}_y$ of $\delta$-separated lines, \eqref{eq:desiredIncidences} follows immediately from the incidence bound in Theorem \ref{mainIncidence} (as pointed out in Remark \ref{rem1}, the theorem remains valid for $C\delta$-incidences, and now we use this with $C = 1 + A$.) \end{proof}

\section{Applications of the Loomis-Whitney inequality in the Heisenberg group}

In this section, we derive the Gagliardo-Nirenberg-Sobolev inequality, Theorem \ref{mainIntro3}, from the Loomis-Whitney inequality, Theorem \ref{mainIntro2}. The arguments presented in this section are very standard, and we claim no originality. As a corollary of Theorem \ref{mainIntro3}, we obtain the isoperimetric inequality in $\He$ (with a non-optimal constant). At the end of the section, we also show how the Loomis-Whitney inequality can be used, directly, to infer a variant of the isoperimetric inequality, without passing through the Sobolev inequality.

We start by recalling the statement of Theorem \ref{mainIntro3}:

\begin{thm}\label{t:sobolev} Let $f \in BV(\He)$. Then,
\begin{equation}\label{eq:GNS} \|f\|_{4/3} \lesssim \sqrt{\|Xf\|\|Yf\|}. \end{equation}
\end{thm}
Recall that $f \in BV(\He)$ if $f \in L^{1}(\He)$, and the distributional derivatives $Xf,Yf$ are finite signed Radon measures. Smooth compactly supported functions are dense in $BV(\He)$ in the sense that if $f \in BV(\He)$, then there exists a sequence $\{\varphi_{j}\}_{j \in \N} \subset C^{\infty}_{c}(\R^{3})$ such that $\varphi_{j} \to f$ almost everywhere (and in $L^{1}(\He)$ if desired), and $\|Z\varphi_{j}\| \to \|Z f\|$ for $Z \in \{X,Y\}$. For a reference, see \cite[Theorem 2.2.2]{MR1437714}. With this approximation in hand, it suffices to prove Theorem \ref{t:sobolev} for, say, $f \in C^{1}_{c}(\R^{3})$. The following lemma contains most of the proof:
\begin{lemma}\label{l:sobolev} Let $f \in C^{1}_{c}(\R^{3})$, and write
\begin{equation}\label{eq:Fk} F_{k} := \{p \in \R^{3} : 2^{k - 1} \leq |f(p)| \leq 2^{k}\}, \qquad k \in \Z. \end{equation}
Then,
\begin{equation}\label{form20} |\pi_{x}(F_{k})| \leq 2^{-k + 2} \int_{F_{k - 1}} |Yf| \quad \text{and} \quad |\pi_{y}(F_{k})| \leq 2^{-k + 2} \int_{F_{k - 1}} |Xf|. \end{equation}
\end{lemma}
\begin{proof} By symmetry, it suffices to prove the first inequality in \eqref{form20}. Let $w = (x,0,t) \in \pi_{x}(F_{k})$, and fix $p = w \cdot (0,y,0) \in F_{k}$ such that $\pi_{x}(p) = w$. In particular, $|f(p)| \geq 2^{k - 1}$. Recall the notation $\mathbb{L}_{y} = \{(0,y,0) : y \in \R\}$. Since $f$ is compactly supported, we may pick another point $p' \in w \cdot \mathbb{L}_{y}$ such that $f(p') = 0$. Since $|f|$ is continuous, we infer that there is a non-degenerate line segment $I$ on the line $w \cdot \mathbb{L}_{y}$ such that $2^{k - 2} \leq |f(q)| \leq 2^{k - 1}$ for all $q \in I$ (hence $I \subset F_{k - 1}$), and $|f|$ takes the values $2^{k - 2}$ and $2^{k - 1}$, respectively, at the endpoints $q_{i} = w \cdot (0,y_{i},0)$ of $I$, $i \in \{1,2\}$. Define $\gamma(y) := w \cdot (0,y,0) = (x,y,t + \tfrac{1}{2}xy)$. With this notation,
\begin{displaymath} 2^{k - 2} \leq |f(q_{1}) - f(q_{2})| \leq \int_{y_{1}}^{y_{2}} |(f \circ \gamma)'(y)| \, dy \leq \int_{\{y : (x,y,t + \frac{1}{2}xy) \in F_{k - 1}\}} |Yf(x,y,t + \tfrac{1}{2}xy)| \, dy. \end{displaymath}
Writing $\Phi(x,y,t) := (x,0,t) \cdot (0,y,0) = (x,y,t + \tfrac{1}{2}xy)$, and integrating over $(x,t) \cong (x,0,t) \in \pi_{x}(F_{k}) \subset \W_{x}$, it follows that
\begin{equation}\label{form21} \int_{\pi_{x}(F_{k})} \left[ \int_{\{y : \Phi(x,y,t) \in F_{k - 1}\}} |Yf(\Phi(x,y,t))| \, dy \right] \, dx \, dt \geq 2^{k - 2}|\pi_{x}(F_{k})|. \end{equation}
Finally, we note that $J_{\Phi} = \mathrm{det\,} D\Phi \equiv 1$. Therefore, using Fubini's theorem, and performing a change of variables to the left hand side of \eqref{form21}, we see that
\begin{align*} 2^{k - 2}|\pi_{x}(F_{k})| & \leq \int_{\{(x,y,t) \in \R^{3} : \Phi(x,y,t) \in F_{k - 1}\}} |Yf(\Phi(x,y,t))| \, dx \, dy \, dt\\
& = \int_{F_{k - 1}} |Yf(x,y,t)| \, dx \, dy \, dt. \end{align*}
This completes the proof. \end{proof}
We are then prepared to prove Theorem \ref{t:sobolev}:
\begin{proof}[Proof of Theorem \ref{t:sobolev}] Fix $f \in C^{1}_{c}(\R^{3})$, and define the sets $F_{k}$, $k \in \Z$, as in \eqref{eq:Fk}. Using first Theorem \ref{main}, then Lemma \ref{l:sobolev}, then Cauchy-Schwarz, and finally the embedding $\ell^{1} \hookrightarrow \ell^{4/3}$, we estimate as follows:
\begin{align*} \int |f|^{4/3} \sim \sum_{k \in \Z} 2^{4k/3}|F_{k}| & \leq \sum_{k \in \Z} 2^{4k/3} |\pi_{x}(F_{k})|^{2/3}|\pi_{y}(F_{k})|^{2/3}\\
& \lesssim \sum_{k \in \Z} \Big(\int_{F_{k - 1}} |Xf| \Big)^{2/3} \Big(\int_{F_{k - 1}} |Yf| \Big)^{2/3}\\
& \leq \Big[ \sum_{k \in \Z} \Big( \int_{F_{k - 1}} |Xf| \Big)^{4/3} \Big]^{1/2} \Big[ \sum_{k \in \Z} \Big( \int_{F_{k - 1}} |Yf| \Big)^{4/3} \Big]^{1/2}\\
& \leq \Big[\sum_{k \in \Z} \int_{F_{k - 1}} |Xf| \Big]^{2/3} \Big[ \sum_{k \in \Z} \int_{F_{k - 1}} |Yf| \Big]^{2/3} = \|Xf\|_{1}^{2/3}\|Yf\|_{1}^{2/3}. \end{align*}
Raising both sides to the power $3/4$ completes the proof. \end{proof}

We conclude the paper by discussing isoperimetric inequalities. A measurable set $E \subset \He$ has \emph{finite horizontal perimeter} if $\chi_{E} \in BV(\He)$. Here $\chi_{E}$ is the characteristic function of $E$. Note that our definition of $BV(\He)$ implies, in particular, that $|E| < \infty$. We follow common practice, and write $P_{\mathbb{H}}(E) := \|\nabla_{\He} \chi_{E}\|$. For more information on sets of finite horizontal perimeter, see \cite{FSSC}. Now, applying Theorem \ref{t:sobolev} to $f = \chi_{E}$, we infer Pansu's isoperimetric inequality (with a non-optimal constant):
\begin{thm}[Pansu] There exists a constant $C>0$ such that
\begin{equation}\label{eq:isop}
|E|^{\frac{3}{4}}\leq C P_{\mathbb{H}}(E)
\end{equation}
for any measurable set $E \subset \He$ of finite horizontal perimeter.
\end{thm}
We remark that the \emph{a priori} assumption $|E| < \infty$ is critical here; for example the theorem evidently fails for $E = \He$, for which $|E| = \infty$ but $\|\nabla_{\He} \chi_{E}\| = 0$. We conclude the paper by deducing a weaker version of \eqref{eq:isop} (even) more directly from the Loomis-Whitney inequality. Namely, we claim that
\begin{equation}\label{weakIP} |E|^{\frac{3}{4}}\leq C \calH^{3}_{d}(\partial E) \end{equation}
for any bounded measurable set $E \subset \He$. This inequality is, in general, weaker than \eqref{eq:isop}: at least for open sets $E \subset \He$, the property $\calH^{3}_{d}(\partial E) < \infty$ implies that $P_{\He}(E) < \infty$, and then $P_{\He}(E) \lesssim \calH^{3}_{d}(\partial E)$, see \cite[Theorem 4.18]{MR2836591}. However, if $E$ is a bounded open set with $C^{1}$ boundary, then $\mathcal{H}^{3}_{d}(\partial E) \sim P_{\He}(E)$, see \cite[Corollary 7.7]{FSSC}.

To prove \eqref{weakIP}, we need the following auxiliary result, see \cite[Lemma 3.4]{MR3992573}:

\begin{lemma} There exists a constant $C > 0$ such that the following holds. Let $\W \subset \mathbb{H}$ be a vertical subgroup. Then,
\begin{equation}\label{per2}
|\pi_{\W}(A)|\leq C \mathcal{H}^3_d(A), \qquad A \subset \He.
\end{equation}
\end{lemma}
\begin{proof}[Proof of \eqref{weakIP}]
Let $E\subset\mathbb{H}$ be bounded and measurable. We first claim that
\begin{align}\label{uguproj}
&\pi_{x}(E)\subseteq \pi_{x}(\partial E),\\
\label{uguproj2}
&\pi_{y}(E)\subseteq \pi_{y}(\partial E)
\end{align}
We prove only \eqref{uguproj} since \eqref{uguproj2} follows similarly. Let $w\in \pi_{x}(E)$ and consider $\pi_{x}^{-1}\{w\} = w \cdot \mathbb{L}_{y}$ where $\mathbb{L}_{y}=\{(0,y,0)\ :\ y\in\mathbb{R}\}$ is as in Section \ref{s:LW}. By definition there exists $y_1\in \mathbb{R}$ such that $w \cdot (0, y_{1}, 0)\in E$ and since $E$ is bounded there also exists $y_2\in\mathbb{R}$ such that $w \cdot  (0,y_{2}, 0)\in \mathbb{H} \, \setminus \, \overline{E}$. Since $w \cdot \mathbb{L}_{y}$ is connected, there finally exists $y_3\in\mathbb{R}$ such that $w \cdot  (0,y_{3},0)\in \partial E$ which immediately implies \eqref{uguproj}. Using Theorem \ref{main}, \eqref{uguproj}, and \eqref{uguproj2} we get
\[
|E| \lesssim |\pi_x(\partial E)|^{\frac{2}{3}}|\pi_y(\partial E)|^{\frac{2}{3}}.
\]
Now the isoperimetric inequality \eqref{weakIP} follows using Lemma \ref{per2}.
\end{proof}

\bibliographystyle{plain}
\bibliography{references}

\end{document}